\documentclass{article}
\usepackage[utf8]{inputenc}

\usepackage{graphicx}

\usepackage{cmap}
\usepackage[linewidth=1pt]{mdframed}
\usepackage[all]{xy}
\usepackage{amsmath,amssymb,amscd,mathrsfs}

\newtheorem{theorem}{Theorem}[section]
\newtheorem{definition}[theorem]{Definition}
\newtheorem{proposition}[theorem]{Proposition}
\newtheorem{lemma}[theorem]{Lemma}

\newtheorem{theoremn}{Main Theorem}
\newenvironment{proofA}{{\noindent\it Proof of Theorem \ref{A}.}\quad}{$\square$}

\title{On localizing subcategories of derived categories of categories of quasi-coherent sheaves on Noetherian algebraic spaces}
\author{Li Lu}

\begin{document}

\maketitle

\section{Introduction}
\subsection{Background}

Classification of localizing subcategories is quite an active subject widely studied by a number of authors (see, for example, \cite{Neeman1992TheDR}, \cite{Hopkins1987GlobalTheory}, \cite{Kanda2015ClassificationSchemes}, \cite{Balmer2005TheCategories}, \cite{Kanda2012ClassifyingSpectrum}, \cite{Takahashi2008ClassifyingRing},
\cite{Tarrio2004BousfieldSchemes}, \cite{Gabriel1962DesAbeliennes},
\cite{Krause2008ThickRings},
\cite{Benson2011StratifyingGroups}
and \cite{Hall2017PerfectStacks}).  In \cite{Neeman1992TheDR}, Neeman and Bokstedt state a remarkable theorem:
\begin{theorem}(Neeman-Bokstedt, 1992)
Let $R$ be a commutative Noetherian ring. We denote by $supp^{-1}(\Phi)$ the subcategory of $\mathbf{D}(R)$ consisting of all $R$-complexes $M^{\bullet}$ with $supp(M^{\bullet}) \subset \Phi$. There is an inclusion-preserving bijection of sets
\begin{eqnarray*}
 \xymatrix{
\{  localizing \ subcategories \ of \ \mathbf{D}(R) \} \ar@< 2pt>[r]^{supp}  & \{  subsets \ of \ Spec(R) \}. \ \ \ \ \ \ \ \ \ \ \ \ \ \ \ \ \ \ \ \ \ \ar[l]^{supp^{-1}}
 }\\
\end{eqnarray*}
\end{theorem}

Neeman and Bokstedt's theorem is a beautiful result. Tarr{\i}o, L{\'o}pez and Salorio in \cite{Tarrio2004BousfieldSchemes} proved a similar result for a Noetherian formal scheme.

\begin{theorem}(Tarr{\i}o-L{\'o}pez-Salorio, \cite{Tarrio2004BousfieldSchemes}, Theorem 4.12, 2004)\label{TLS1}
For a Noetherian formal scheme $\mathfrak{X}$ there is a bijection between the class of rigid localizing subcategories of $\mathbf{D}_{qct}(\mathfrak{X})$ and the set of all subsets of $\mathfrak{X}$.
\end{theorem}

Our aim is to prove a similar result for a Noetherian separated algebraic space.

\subsection{The Main theorem \ref{gongzuokuang}}

\begin{definition}
Let $S$ be a scheme contained in $\mathbf{Sch}_{fppf}$. Let $\mathfrak{X}$ be an algebraic space over $S$.  
\begin{itemize}
\item We say $\mathfrak{X}$ is a \textit{locally Noetherian algebraic space} if for every scheme $U$ and every $\acute{e}$tale morphism $U\rightarrow \mathfrak{X}$ the scheme $U$ is locally Noetherian. 

\item We say $\mathfrak{X}$ is \textit{quasi-separated} over $S$ if the diagonal morphism $\mathfrak{X}\rightarrow \mathfrak{X}\times_S \mathfrak{X}$ is quasi-compact. (A morphism of algebraic spaces is called quasi-compact if the underlying map of topological spaces is quasi-compact. We say that a continuous map $f:X \rightarrow Y$ is quasi-compact if the inverse image $f^{-1}(V)$ of every quasi-compact open $V\subset Y$ is quasi-compact.)

\item We say $\mathfrak{X}$ is \textit{Noetherian} if $\mathfrak{X}$ is quasi-compact, quasi-separated and locally Noetherian.

\item We say $\mathfrak{X}$ is \textit{separated} over $S$ if the diagonal morphism $\mathfrak{X}\rightarrow \mathfrak{X}\times_S \mathfrak{X}$ is a closed immersion.
\end{itemize}
\end{definition}
Let $S$ be a scheme contained in $\mathbf{Sch}_{fppf}$. Let $\mathfrak{X}$ be a Noetherian separated algebraic space over $S$. For a subset $W$ of $|\mathfrak{X}|$, we denote by $supp^{-1}(W)$ the subcategory of $\mathbf{D}(QCoh(\mathfrak{X}))$ consisting of all complexes $\mathcal{F}^{\bullet}$ of quasi-coherent sheaves on $\mathfrak{X}$ with $supp(\mathcal{F}^{\bullet}) \subset W$. We define the subcategory $\mathcal{L}_W$ as the smallest localizing subcategory of $\mathbf{D}(QCoh(\mathfrak{X}))$ that contains the set  $\{j_{\overline{x}*}\kappa(\overline{x}): x \in W\}$. Here is our first result:
\begin{theoremn}(Theorem \ref{miao})\label{gongzuokuang}
Let $S$ be a scheme contained in $\mathbf{Sch}_{fppf}$. Let $\mathfrak{X}$ be a Noetherian separated algebraic space over $S$. Let $W \subset |\mathfrak{X}|$ be a subset. Then the following categories are the same:
\begin{itemize}
\item (1) $supp^{-1}(W).$
\item (2) $\mathcal{L}_W.$
\item (3) The full subcategory of complexes that are quasi-isomorphic to injective complexes whose components are direct sums of $j_{\overline{x}*}E(\overline{x})$ with $x \in W$.
\end{itemize}
\end{theoremn}
\begin{remark}
According to Lemma \ref{xiaojieyinliya}, they are all rigid  localizing subcategories  of $\mathbf{D}(QCoh(\mathfrak{X}))$.
\end{remark}

\subsection{The Main theorem \ref{xiaojie}}

We will establish a bijective correspondence between the class of rigid localizing subcategories of $\mathbf{D}(QCoh(\mathfrak{X}))$ and the class of $E$-stable subcategories of $QCoh(\mathfrak{X})$ closed under direct  sums  and  summands.

\begin{theoremn}(Theorem \ref{youyixiaojie}, Theorem \ref{yuzhouxiaojie})\label{xiaojie}
Let $S$ be a scheme contained in $\mathbf{Sch}_{fppf}$. Let $\mathfrak{X}$ be a Noetherian separated algebraic space over $S$. 
Consider the following three sets:
\begin{itemize}
\item (1) $\mathbb{LS}=\{ Rigid \ localizing\  subcategories \ of \   \mathbf{D}(QCoh(\mathfrak{X})) \}$.
\item (2) $\mathbb{SX}=\{ Subsets\  of \  |\mathfrak{X}| \}$.
\item (3) $\mathbb{ES}=\left\{
\begin{aligned}
E-stable \ subcategories \ of \ QCoh(\mathfrak{X}) \ closed\\
under \ direct\  sums\  and\  summands
\end{aligned}
\right\}.$
\end{itemize}

All of the three sets are bijectively corresponding to one another:

\begin{eqnarray*}
 \xymatrix{
  \mathbb{LS} \ar[rr]_{supp^{-1}} &  &  \mathbb{SX} \ar@< 2pt>[ll]_{supp} \ar[rr]_{supp} &  &  \mathbb{ES}. \ar@< 2pt>[ll]_{(supp^{-1}(-))_0}
   }\\
    \end{eqnarray*}   
\end{theoremn}

\subsection{Symbols}

\begin{tabular}{cc}
\hline
$\mathbf{Sch}_{fppf}$ &  big $fppf$ site  \\
$S$ &  scheme \\
$\mathfrak{X}$ &  algebraic space over $S$ \\
$|\mathfrak{X}|$ &  underlying topological space of $\mathfrak{X}$ \\
$x \in |\mathfrak{X}|$ &  point of $\mathfrak{X}$ \\
$\overline{x}$ &  geometric point lying over $x$ \\
$QCoh(\mathfrak{X})$ & category of quasi-coherent $O_{\mathfrak{X}}$-sheaves \\
$Coh(\mathfrak{X})$ & category of coherent $O_{\mathfrak{X}}$-sheaves \\
$\mathcal{A}$ &  Abelian category \\
$\mathbf{C}(\mathcal{A})$ &  category of $\mathcal{A}$-complexes \\
$\mathbf{K}(\mathcal{A})$ &  homotopy category of chain complexes in $\mathcal{A}$ \\
$\mathbf{D}(\mathcal{A})$ &  derived category of $\mathcal{A}$ \\
\hline
\end{tabular}

\subsection{Brief outline}

We now give a brief outline of the paper. 

In section \ref{Preliminaries}, we set notation and review some basics of localizing subcategories, torsion theories, $t$-structures, truncation of complexes, homotopy colimits and injective quasi-coherent sheaves.

Section \ref{Local cohomology} is concerned with the local cohomology functors. The local cohomology functor is a helpful tool for the classification of all the rigid localizing subcategories of $\mathbf{D}(QCoh(\mathfrak{X}))$. 

In Section \ref{Small supports}, we will  study basic properties of the small supports.

In Section \ref{zhuyaozhangjie1}, we will prove the Main Theorem \ref{gongzuokuang} and the Main Theorem \ref{xiaojie}.

\section{Preliminaries}
\label{Preliminaries}

\subsection{Localizing subcategories of Grothendieck categories}

Our main references for categories are \cite{Gabriel1962DesAbeliennes}, 
\cite{Grothendieck1957SurI},
\cite{MacLane1963NaturalCommutativity},
\cite{Freyd1964AbelianCategories}.

Let $\mathcal{A}$ be an Abelian category, and let $\mathcal{S}$ be a nonempty full subcategory of $\mathcal{A}$. $\mathcal{S}$ is a Serre subcategory provided that it is closed under extensions, subobjects, and quotient objects. 

Let $\mathcal{A}$ be an Abelian category. Let $\mathcal{S}$ be a Serre subcategory. There exists an Abelian category $\mathcal{A}/\mathcal{S}$ and an exact functor $F:\mathcal{A} \rightarrow \mathcal{A} / \mathcal{S}$, which is essentially surjective and whose kernel is $\mathcal{S}$. The category $\mathcal{A}/\mathcal{S}$ and the functor $F$ are characterized by the following universal property: For any exact functor $G:\mathcal{A} \rightarrow \mathcal{B}$ such that $\mathcal{S} \subset Ker(G)$ there exists a factorization $G=H\circ F$ for a unique exact functor $H:\mathcal{A}/\mathcal{S} \rightarrow \mathcal{B}$.

The Serre subcategory $\mathcal{S}$ is called localizing, if the quotient functor $F:\mathcal{A} \rightarrow \mathcal{A} / \mathcal{S}$ has a right adjoint $G:\mathcal{A} / \mathcal{S} \rightarrow \mathcal{A}$. 

A Grothendieck category is an Abelian category which has coproducts, in which direct limits are exact and which has a generator. A category $\mathcal{A}$ is well-powered if for each $A \in Obj(\mathcal{A})$, the class of subobjects of $A$ is a set. Grothendieck categories are well-powered. Let $\mathcal{A}$ be a Grothendieck category, then $\mathcal{A}$ has products and $\mathcal{A}$ has enough injectives (See \cite{Freyd1964AbelianCategories}, Theorem 6.25).

If $\mathcal{A}$ is a Grothendieck category, then a Serre subcategory $\mathcal{S}$ is localizing if and only if $\mathcal{S}$ is closed under arbitrary coproducts. If $\mathcal{A}$ is a Grothendieck category and $\mathcal{S}$ a localizing subcategory, then the quotient category $\mathcal{A} / \mathcal{S}$ is again a Grothendieck category.

\subsection{Localizing subcategories of triangulated categories}

For the definition of the triangulated category see, for example, \cite{Verdier1996DesAbeliennes}, \cite{Neeman2014TriangulatedCategories.AM-148}, \cite{Weibel2013AnAlgebra}. 

\begin{definition}
Let $\mathcal{D}$ be a triangulated category and let $\mathcal{X}$ be a full additive subcategory of $\mathcal{D}$.
\begin{itemize}
\item(1) $\mathcal{X}$ is called a triangulated subcategory if every object isomorphic to an object of $\mathcal{X}$ is in $\mathcal{X}$, if $\mathcal{X}[1] = \mathcal{X}$, and if for any distinguished triangle
\begin{equation}
X \rightarrow Y \rightarrow Z \rightarrow X[1]
\end{equation}
such that the objects $X$ and $Y$ are in $\mathcal{X}$, the object $Z$ is also in $\mathcal{X}$.
\item(2) We say that $\mathcal{X}$ is thick if $\mathcal{X}$ is triangulated and closed under direct summands.
\item(3) We say that $\mathcal{X}$ is localizing if $\mathcal{X}$ is triangulated and closed under arbitrary direct sums.
\end{itemize}
\end{definition}
\begin{remark}
Localizing subcategories of $\mathcal{D}$ are closed under direct summands (see \cite{Zhang2015TriangulatedCategories}, Proposition 3.6.1).
\end{remark}

\subsection{Torsion theories}
Our main references for the torsion theories are \cite{Dickson1966ACategories}, \cite{Popescu1973AbelianModules}, \cite{Lamber2006TorsionQuotients}.

A torsion theory in a Grothendieck category $\mathcal{A}$ is a couple $(\mathcal{T},\mathcal{F})$ of strictly full additive subcategories called the torsion class $\mathcal{T}$ and the torsion free class $\mathcal{F}$ such that the following conditions hold:
\begin{itemize}
\item (1) $Hom(\mathcal{T},\mathcal{F})=0.$
\item (2) For all $M \in Obj(\mathcal{A})$, there exists $N\subset M$, $N\in Obj(\mathcal{T})$ and $M/N\in Obj(\mathcal{F}).$
\end{itemize}

For every object $M$ there exist the largest subobject $t(M)\subset M$ which is in $\mathcal{T}$ and it is called the torsion part of $M$. 
\begin{equation}\label{daoyichuan}
t:M\mapsto t(M)
\end{equation}
is an additive functor. A torsion theory is hereditary if $\mathcal{T}$ is closed under subobjects, or equivalently, $t$ is left exact functor.

A radical functor, or more generally a preradical functor, has its own long history in the theory of categories and functors. See \cite{Goldman1969RingsQuotients} or \cite{Maranda1964InjectiveStructures} for the case of module category. Let $F, G: \mathcal{A}\rightarrow \mathcal{A}$ be functors. Recall that $F$ is said to be a subfunctor of $G$, denoted by $F \subset G$, if $F(M)$ is a subobject of $G(M)$ for all $M \in \mathcal{A}$ and if $F(f)$ is a restriction of $G(f)$ to $F(M)$ for all $f \in Hom_{\mathcal{A}}(M, N)$.  

\begin{definition}\label{preradical}
A functor $F:\mathcal{A}\rightarrow \mathcal{A}$ is called a preradical functor if $F$ is a subfunctor of $\mathbf{1}$.
\end{definition}
\begin{lemma}(\cite{Yoshino2011AbstractFunctors}, Lemma 1.1)\label{prehanzijiao}
Let $F: \mathcal{A} \rightarrow \mathcal{A}$ be a preradical functor and assume that $F$ is a left exact functor on $\mathcal{A}$. If $N$ is a subobject of $M$, then the equality $F(N) = N \cap F(M)$ holds.
\end{lemma}

\begin{definition}\label{radical}
A preradical functor $F$ is called a radical functor if $F(M/F(M)) = 0$ for all $M \in Obj(\mathcal{A})$. 
\end{definition}
If $F: \mathcal{A}\rightarrow \mathcal{A}$ is a left exact radical functor, then there is a hereditary torsion theory $(\mathcal{T}_F, \mathcal{F}_F)$ by setting
\begin{equation}
\mathcal{T}_F = \{ M \in \mathcal{A} | F(M) = M \},  \nonumber
\end{equation}
\begin{equation}\label{yichuan}
\mathcal{F}_F = \{ M \in \mathcal{A} | F(M) = 0 \}.
\end{equation}

\subsection{$t$-structures}
The notion of a $t$-structure arose in the work \cite{Beilinson1982FaisceauxI} of Beilinson, Bernstein, Deligne, and Gabber on perverse sheaves.

Let $\mathcal{D}$ be a triangulated category. $t$-structure in $\mathcal{D}$ is a pair $\mathbf{t} = (\mathcal{U}, \mathcal{W})$ of full subcategories, closed under taking direct summands in $\mathcal{D}$, which satisfy the following properties:
\begin{itemize}
\item (t-S.1) $Hom_\mathcal{D}(U, W[-1])=0$, for all $U \in \mathcal{U}$ and $W \in \mathcal{W}$; 
\item (t-S.2) $\mathcal{U}[1]\subset \mathcal{U},$ $\mathcal{W}[-1]\subset \mathcal{W}$;
\item (t-S.3) for each $Y \in \mathcal{D}$, there is a triangle $A \rightarrow Y \rightarrow B \rightarrow A[1]$ in $\mathcal{D}$, where $A \in \mathcal{U}$ and $B \in \mathcal{W}[-1]$.
\end{itemize}

A $t$-structure $\mathbf{t} = (\mathcal{U}, \mathcal{W})$ in $\mathcal{D}$ is called a stable $t$-structure on $\mathcal{D}$ if $\mathcal{U}$ and $\mathcal{W}$ are triangulated subcategories.

\begin{theorem}(\cite{Miyachi1991LocalizationCategories}, Proposition 2.6)\label{miqi}
Let $\mathcal{D}$ be a triangulated category and $\mathcal{U}$ a triangulated subcategory of $\mathcal{D}$. Then the following conditions are equivalent for $\mathcal{U}$.
\begin{itemize}
\item (1) There is a triangulated subcategory $\mathcal{W}$ of $\mathcal{D}$ such that $(\mathcal{U}, \mathcal{W})$ is a stable $t$-structure on $\mathcal{D}$.
\item (2) The natural embedding functor $i : \mathcal{U} \rightarrow \mathcal{D}$ has a right adjoint $\rho : \mathcal{D} \rightarrow \mathcal{U}$. 
\end{itemize}
If it is the case, setting $\delta=i\circ\rho:\mathcal{D} \rightarrow \mathcal{D}$,we have the equalities $\mathcal{U} =Im(\delta)$ and $\mathcal{W} =Ker(\delta)$. There is an natural morphism $\phi: \delta \rightarrow \mathbf{1}$, where $\mathbf{1}$ is the identity functor on $\mathcal{D}$. Every $C \in \mathcal{D}$ can be embedded in a triangle of the form
\begin{equation}
\delta(C) \xrightarrow{\phi(C)} C \rightarrow D \rightarrow \delta(C)[1].
\end{equation}
\end{theorem}

\subsection{Truncation of complexes}
Let $\mathcal{A}$ be an Abelian category. Let $M^{\bullet}$ be a chain complex. There are several ways to truncate the complex $M^{\bullet}$:
\begin{itemize}
\item (1) The left brutal truncation $\sigma_{\leq n}$ is the subcomplex $\sigma_{\leq n}M^{\bullet}$ defined by the rule 
\begin{equation}
(\sigma_{\leq n}M^{\bullet})^i=\left\{
\begin{aligned}
0 \ (i>n)\\
M^i \  (i\leq n).
\end{aligned}
\right.
\end{equation}
\item (2) The right brutal truncation $\sigma_{\geq n}$ is the subcomplex $\sigma_{\geq n}M^{\bullet}$ defined by the rule 
\begin{equation}
(\sigma_{\geq n}M^{\bullet})^i=\left\{
\begin{aligned}
0 \ (i<n)\\
M^i \  (i\geq n).
\end{aligned}
\right.
\end{equation}
\item (3) The left good truncation $\tau_{\leq n}$ is the subcomplex $\tau_{\leq n}M^{\bullet}$ defined by the rule
\begin{equation}
(\tau_{\leq n}M^{\bullet})^i=\left\{
\begin{aligned}
0 \ (i>n)\\
Kerd^n \ (i=n) \\
M^i \  (i< n).
\end{aligned}
\right.
\end{equation}
\item (4) The right good truncation $\tau_{\geq n}$ is the subcomplex $\tau_{\geq n}M^{\bullet}$ defined by the rule 
\begin{equation}
(\tau_{\geq n}M^{\bullet})^i=\left\{
\begin{aligned}
0 \ (i\geq n)\\
Imd^{n-1} \ (i=n-1) \\
M^i \  (i< n-1).
\end{aligned}
\right.
\end{equation}
\end{itemize}
\begin{lemma}[\cite{Zhang2015TriangulatedCategories}, Lemma 2.6.1]\label{jieduanzhenghe} Let $\mathcal{A}$ be an Abelian category. Let $M^{\bullet}$ be a chain complex. Then
\begin{equation}
0\rightarrow \sigma_{\geq n} M^{\bullet} \rightarrow M^{\bullet}\rightarrow \sigma_{\leq n-1} M^{\bullet} \rightarrow 0
\end{equation}
is an exact sequence.
\end{lemma}

\begin{lemma}[\cite{Zhang2015TriangulatedCategories}, Lemma 2.6.2]\label{shuangjieduan} Let $\mathcal{A}$ be an Abelian category. Let $M^{\bullet}$ be a chain complex. Then
\begin{equation}
colim_{n\geq 0} [\sigma_{\geq -n} (\tau_{\leq n+1}M^{\bullet})] = M^{\bullet}.
\end{equation}
\end{lemma}

\subsection{Homotopy colimits}

Our main references for the homotopy colimits are \cite{Bokstedt1993HomotopyCategories}.
\begin{lemma}[\cite{Zhang2015TriangulatedCategories}, Proposition 1.3.4]
Let $\mathcal{D}$ be a triangulated category. Let $I$ be a set.
\begin{itemize}
\item (1) Let $\{ X_i, i\in I \}$ be a family of objects of $\mathcal{D}$. If $\bigoplus X_i$ exists, then $(\bigoplus X_i)[1]=\bigoplus X_i[1]$.
\item (2) Let $X_i\rightarrow Y_i\rightarrow Z_i\rightarrow X_i[1]$ be a family of distinguished triangles of $\mathcal{D}$. If $\bigoplus X_i,$ $\bigoplus Y_i,$ $\bigoplus Z_i$ exist, then
$\bigoplus X_i\rightarrow \bigoplus Y_i\rightarrow \bigoplus Z_i\rightarrow \bigoplus X_i[1]$ is a distinguished triangle.
\end{itemize}
\end{lemma}

\begin{lemma}[\cite{Bokstedt1993HomotopyCategories}, Lemma 1.1]
Let $\mathcal{A}$ be a Grothendieck category. Write $\mathbf{K}(\mathcal{A})$ for the homotopy category of chain complexes in $\mathcal{A}$. Then $\mathbf{K}(\mathcal{A})$ has direct sums.
\end{lemma}

\begin{definition}
Let $\mathcal{D}$ be a triangulated category with direct sums. Suppose $\{ X_i, i\geq 0 \}$ is a sequence of objects in $\mathcal{D}$, together with maps $f_i: X_i \rightarrow X_{i+1}$. We say an object $K$ is a derived colimit, or a homotopy colimit of the system $\{ X_i, i\geq 0 \}$ if there is a distinguished triangle
\begin{equation}
\bigoplus X_i \rightarrow \bigoplus X_i\rightarrow K \rightarrow \bigoplus X_i[1]
\end{equation}
where the map $\bigoplus X_i \rightarrow \bigoplus X_i$ is given by $1-f_n$ in degree $n$. If this is the case, then we sometimes indicate this by the notation $K=hocolim(X_i)$.
\end{definition}

\begin{lemma}[\cite{Bokstedt1993HomotopyCategories}, Remark 2.2]\label{holim}
If $\mathcal{D} = \mathbf{D}(\mathcal{A})$, and $\mathcal{A}$ is a Grothendieck category, then
\begin{equation}
\mathbf{H}^i(hocolim(X_j)) = colim \mathbf{H}^i(X_j).
\end{equation}
\end{lemma}

\subsection{Injective quasi-coherent sheaves}

Let $S$ be a scheme. Let $\mathfrak{X}$ be a locally Noetherian algebraic space over $S$. Let $\mathcal{F}$ be a quasi-coherent $\mathcal{O}_{\mathfrak{X}}$-sheaf. For each $x \in |\mathfrak{X}|$, denote the unique maximal ideal of $\mathcal{O}_{\mathfrak{X},\overline{x}}$ by $\mathfrak{m}_{\overline{x}}$, the residue field of $\overline{x}$ by $\kappa(\overline{x}) = \mathcal{O}_{\mathfrak{X},\overline{x}}/\mathfrak{m}_{\overline{x}}$, and an injective hull of $\kappa(\overline{x})$ in $Mod(\mathcal{O}_{\mathfrak{X},\overline{x}})$ by $E(\overline{x}) = E_{\mathcal{O}_{\mathfrak{X},\overline{x}}}(\kappa(\overline{x}))$. Let $j_{\overline{x}} : Spec(\mathcal{O}_{\mathfrak{X},\overline{x}}) \rightarrow \mathfrak{X}$ be the canonical morphism.
We state that every injective quasi-coherent sheaf is a direct sum of indecomposable injective quasi-coherent sheaves of this form.
\begin{theorem}
\label{neishefenjie}
Let $\mathfrak{X}$ be a locally Noetherian algebraic space over $S$.
\begin{itemize}
 \item (1) For every family $\{ \mathcal{J}_{\lambda} \}_{\lambda \in \Lambda}$ of injective quasi-coherent sheaves, the direct sum $\bigoplus_{\lambda \in \Lambda} \mathcal{J}_{\lambda}$ is also injective.
\item (2) Every injective quasi-coherent sheaf has an indecomposable decomposition.
\item (3) For each $x \in |\mathfrak{X}|$ , let $j_{\overline{x}} : Spec(\mathcal{O}_{\mathfrak{X},\overline{x}}) \rightarrow \mathfrak{X}$ be the canonical morphism. There is a bijection
\begin{equation}
|\mathfrak{X}| \rightarrow \frac{\{ \ indecomposable\  injective\  quasi-coherent\  sheaves \}}{\cong}
\end{equation}
given by
\begin{equation}
x \mapsto j_{\overline{x}*}E(\overline{x}).
\end{equation}
\end{itemize}
\end{theorem}
\begin{proof}
(1) See \cite{Kanda2012ClassifyingSpectrum}, Theorem 5.9.

(2) See \cite{Kanda2012ClassifyingSpectrum}, Theorem 5.9.

(3) Choose a scheme $U$ with a point $u$ and an $\acute{e}$tale morphism $j_U: U\rightarrow X$ mapping $u$ to $x$, then we have $j_U^{*}j_{\overline{x}*}E(\overline{x}) = j_{u*}E(u)$ is an indecomposable injective quasi-coherent $\mathcal{O}_U$-sheaf by \cite{Hartshorne1966ResiduesDuality} Theorem 7.11. Hence $j_{\overline{x}*}E(\overline{x})$ is an indecomposable injective quasi-coherent $\mathcal{O}_{\mathfrak{X}}$-sheaf.

Now for an indecomposable injective quasi-coherent $\mathcal{O}_{\mathfrak{X}}$-sheaf $\mathcal{J}$, we have $j_U^{*}\mathcal{J}$ is an indecomposable injective quasi-coherent $\mathcal{O}_{\mathfrak{X}}$-sheaf, hence $j_U^{*}\mathcal{J} = j_{u*}E(u)$ for some $u\in U$, we have $j_{\overline{x}*}E(\overline{x}) \subset \mathcal{J}.$ Since $\mathcal{J}$ is an indecomposable injective quasi-coherent $\mathcal{O}_{\mathfrak{X}}$-sheaf, we have $\mathcal{J}= j_{\overline{x}*}E(\overline{x}).$
\end{proof}

Let $\mathcal{A}$ be an Abelian category. A complex $J^{\bullet}$ is $K$-injective if for every acyclic complex $M^{\bullet}$ we have $Hom_{\mathbf{K}(\mathcal{A})}(M^{\bullet}, J^{\bullet})=0$.
An important property of $K$-injective objects is that
\begin{equation}
Hom_{\mathbf{K}(\mathcal{A})}(M^{\bullet}, J^{\bullet})=Hom_{\mathbf{D}(\mathcal{A})}(M^{\bullet}, J^{\bullet}),
\end{equation}
for every $M^{\bullet}\in \mathbf{K}(\mathcal{A})$. A $K$-injective complex $J^{\bullet}$ together with a quasi-isomorphism $M^{\bullet} \rightarrow J^{\bullet}$ is called a minimal $K$-injective resolution of $M^{\bullet}$, if for all $i$ the kernel of the differential $J^i \rightarrow J^{i+1}$ is an essential subobject of $J^i$.

\begin{lemma}[\cite{Tarrio2000LocalizationResolutions}, Theorem 5.4; \cite{Krause2005TheScheme}, Proposition B.2]
\label{daochudengjia}
Let $\mathfrak{X}$ be a Noetherian separated algebraic space over $S$, and let $QCoh(\mathfrak{X})$ be the category of quasi-coherent sheaves on $\mathfrak{X}$. In the category $\mathbf{C}(QCoh(\mathfrak{X}))$ of quasi-coherent $\mathcal{O}_{\mathfrak{X}}$-complexes, every object has a minimal $K$-injective resolution.
\end{lemma}

\section{Local cohomology}\label{Local cohomology}

Let $S$ be a scheme contained in $\mathbf{Sch}_{fppf}$. Let $\mathfrak{X}$ be a Noetherian separated algebraic space over $S$. Let $\mathcal{F}$ be a quasi-coherent sheaf on $\mathfrak{X}$. For each $x \in |\mathfrak{X}|$, denote the unique maximal ideal of $\mathcal{O}_{\mathfrak{X},\overline{x}}$ by $\mathfrak{m}_{\overline{x}}$, the residue field of $\overline{x}$ by $\kappa(\overline{x}) = \mathcal{O}_{\mathfrak{X},\overline{x}}/\mathfrak{m}_{\overline{x}}$, and an injective hull of $\kappa(\overline{x})$ in $Mod(\mathcal{O}_{\mathfrak{X},\overline{x}})$ by $E(\overline{x}) = E_{\mathcal{O}_{\mathfrak{X},\overline{x}}}(\kappa(\overline{x}))$. 
\begin{definition}
We say $x \in |\mathfrak{X}|$ is associated to $\mathcal{F}$ if the maximal ideal $\mathfrak{m}_{\overline{x}}$ is associated to the $\mathcal{O}_{\mathfrak{X},\overline{x}}$-module $\mathcal{F}_{\overline{x}}$. We denote $Ass_{\mathfrak{X}}(\mathcal{F})$ the set of associated points of $\mathcal{F}$.
\end{definition}

Let $\mathcal{F} \in  QCoh(\mathfrak{X})$. Following \cite{Foxby1979BoundedModules}, the small support of $\mathcal{F}$ is by definition 
\begin{equation}
supp(\mathcal{F}) = \{x \in X | Tor^{\mathcal{O}_{\mathfrak{X},\overline{x}}}_{*} (\mathcal{F}_{\overline{x}},\kappa(\overline{x})) \not= 0 \}.
\end{equation}

The (usual) support of $\mathcal{F}$ is by definition 
\begin{equation}
Supp(\mathcal{F}) = \{x \in X | \mathcal{F}_{\overline{x}} \not= 0 \} .
\end{equation}

Note that $supp(\mathcal{F}) \subset Supp(\mathcal{F})$ and equality holds if $\mathcal{F} \in  Coh(\mathfrak{X})$ (See \cite{Foxby1979BoundedModules}, Lemma 2.6).

\begin{proposition}\label{ass}
Let $\mathfrak{X}$ be a Noetherian separated algebraic space over $S$ and let $x \in |\mathfrak{X}|$, then $Ass_{\mathfrak{X}}(j_{\overline{x}*}\kappa(\overline{x}))= \{x \},$ $Ass_{\mathfrak{X}}(j_{\overline{x}*}E(\overline{x}))= \{x \}.$
\end{proposition}
\begin{proof}
Let $R$ be a Noetherian ring and let $\mathfrak{p} \in Spec(R)$, then $Ass(R/\mathfrak{p})= \{\mathfrak{p} \}$ and $Ass(E(R/\mathfrak{p}))= \{\mathfrak{p} \}$ (See \cite{Bourbaki1989CommutativeAlgebra} Proposition 4.1.1). Therefore $(j_{\overline{x}*}E(\overline{x}))_{\overline{x}}=j_{\overline{x}}^{-1}j_{\overline{x}*}E(\overline{x})=E(\overline{x})=E_{\mathcal{O}_{\mathfrak{X},\overline{x}}}(\mathcal{O}_{\mathfrak{X},\overline{x}}/\mathfrak{m}_{\overline{x}})$.  Hence $Ass_{\mathfrak{X}}(j_{\overline{x}*}E(\overline{x}))= \{x \}.$ 
\end{proof}

\begin{definition}
For any subset $V \subset |\mathfrak{X}|$ we say that $V$ is a specialization-closed subset if for any $x \in V$ and any $y \in |\mathfrak{X}|$ we have $y\in V$ whenever $y\in \overline{\{x \}}$.
\end{definition}
\begin{definition}
Let $V$ be a specialization-closed subset of $|\mathfrak{X}|$. We can define the section functor
$\Gamma_{V}$ with support in $V$ as 
\begin{equation}\nonumber
\Gamma_V (\mathcal{F}) = \bigcup \{ \mathcal{G} \subset \mathcal{F} | Supp(\mathcal{G}) \subset V \} = \bigcup \{ \mathcal{G} \subset \mathcal{F} | supp(\mathcal{G}) \subset V \}
\end{equation}
for all $\mathcal{F} \in  QCoh(\mathfrak{X})$. 
\end{definition}

\begin{theorem}\label{A}
The following conditions are equivalent for a left exact preradical
functor $F$ on $QCoh(\mathfrak{X})$.
\begin{itemize}
\item (1) $F$ is a radical functor.
\item (2) $F$ preserves injectivity.
\item (3) $F$ is a section functor with support in a specialization-closed subset of $|\mathfrak{X}|$.
\end{itemize}
\end{theorem}

The proof of Theorem \ref{A} consists of a succession of relatively short lemmas. 

\begin{lemma}\label{jieying}
Let $\mathfrak{X}$ be a Noetherian separated algebraic space over $S$, and $V$ be a specialization-closed subset of $|\mathfrak{X}|$.
\begin{itemize}
\item (1) If $\mathcal{G}$ is a subsheaf of $\mathcal{F}$, then the equality $\Gamma_{V}(\mathcal{G}) = \mathcal{G} \cap \Gamma_{V}(\mathcal{F})$ holds.
\item (2) $\Gamma_V(\mathcal{F}/\Gamma_V(\mathcal{F})) = 0$ for every $\mathcal{F}\in QCoh(\mathfrak{X})$.
\item (3) $\Gamma_V$ is a left exact radical functor. 
\end{itemize}
\end{lemma}

\begin{proof}
(1) If $\mathcal{H}$ is a quasi-coherent sheaf on $\mathfrak{X}$, then
\begin{equation}\nonumber
\mathcal{H}\subset \Gamma_V(\mathcal{G})\Leftrightarrow \mathcal{H}\subset \mathcal{G}\  and\  Supp(\mathcal{H})\subset V \Leftrightarrow \mathcal{H}\subset \Gamma_V(\mathcal{F}) \cap \mathcal{G}.
\end{equation}

(2) If $\mathcal{H}$ is a quasi-coherent sheaf on $\mathfrak{X}$, then
\begin{equation}\nonumber
\mathcal{H}\subset \Gamma_V(\mathcal{F}/\Gamma_V(\mathcal{F}))\Leftrightarrow \mathcal{H}\subset \mathcal{F}/\Gamma_V(\mathcal{F})\  and\  Supp(\mathcal{H})\subset V \Rightarrow \mathcal{H}=0.
\end{equation}

(3) Let $0\rightarrow \mathcal{K} \xrightarrow{f} \mathcal{F} \xrightarrow{g} \mathcal{G}$ be an exact sequence in $QCoh(\mathfrak{X})$. By (1) we have $\Gamma_{V}(\mathcal{K}) = \mathcal{K} \cap \Gamma_{V}(\mathcal{F})$, hence $0\rightarrow \Gamma_V(\mathcal{K}) \rightarrow \Gamma_V(\mathcal{F})$ is an exact sequence. Let $\mathcal{H}$ be a quasi-coherent sheaf on $\mathfrak{X}$. We have
\begin{equation}\nonumber
\mathcal{H} \subset Ker\Gamma_V(g)\Leftrightarrow \mathcal{H}\subset Kerg\cap\Gamma_V(\mathcal{F})= Imf\cap\Gamma_V(\mathcal{F})\Leftrightarrow \mathcal{H}\subset Im\Gamma_V(f), 
\end{equation}
hence $0\rightarrow \Gamma_V(\mathcal{K}) \rightarrow \Gamma_V(\mathcal{F})  \rightarrow \Gamma_V(\mathcal{G})$ is an exact sequence.
\end{proof}

\begin{lemma}
Let $\mathfrak{X}$ be a Noetherian algebraic space over $S$. Let $F: QCoh(\mathfrak{X})\rightarrow QCoh(\mathfrak{X})$ be a left exact radical functor.\label{radicbaoneishe}
\begin{itemize}
\item (1) Let $x \in |\mathfrak{X}|$, then $F(j_{\overline{x}*}E(\overline{x}))$ is identical to either $j_{\overline{x}*}E(\overline{x})$ or $0$.
\item (2) $F$ preserves injectivity.
\end{itemize}
\end{lemma}
\begin{proof}
(1) Since $F$ is a left exact radical functor, there is a hereditary torsion theory $(\mathcal{T}_{F}, \mathcal{F}_{F})$, which is defined in (\ref{yichuan}). Then there is an exact sequence 
\begin{equation}\nonumber
0 \rightarrow \mathcal{G} \rightarrow j_{\overline{x}*}E(\overline{x}) \rightarrow \mathcal{H}  \rightarrow 0 
\end{equation}
with $\mathcal{G} \in \mathcal{T}_{F}$ and $\mathcal{H} \in \mathcal{F}_{F}$. 
If $\mathcal{G} = 0$, then $j_{\overline{x}*}E(\overline{x}) \cong \mathcal{H} \in \mathcal{F}_{F}$, therefore $F(j_{\overline{x}*}E(\overline{x})) = 0$. If $\mathcal{G} \not= 0$, since $Ass_{\mathfrak{X}}(j_{\overline{x}*}E(\overline{x}))= x$, we have $Ass_{\mathfrak{X}}(\mathcal{G})= x$, hence $j_{\overline{x}*}\kappa(\overline{x}) \subset \mathcal{G}.$ Since $\mathcal{T}_{F}$ is a localizing subcategory and $\mathcal{G} \in \mathcal{T}_{F}$, we have $j_{\overline{x}*}\kappa(\overline{x}) \in \mathcal{T}_{F}.$ By Proposition \ref{ass} $j_{\overline{x}*}E(\overline{x}) \in \mathcal{T}_{F}.$ Therefore $F(j_{\overline{x}*}E(\overline{x})) = j_{\overline{x}*}E(\overline{x})$. 

(2) For an injective sheaf $\mathcal{J} \in QCoh(\mathfrak{X})$, by Theorem \ref{neishefenjie} it has an indecomposable decomposition $\mathcal{J} = \bigoplus_{i\in \mathcal{I}}j_{\overline{x_i}*}E(\overline{x_i})$. We set $\mathcal{J}_1 = \bigoplus_{i\in \mathcal{I}_1} j_{\overline{x_i}*}E(\overline{x_i})$ and $\mathcal{J}_2 = \bigoplus_{i\in \mathcal{I}_2} j_{\overline{x_i}*}E(\overline{x_i})$, where
\begin{equation}
\mathcal{I}_1 = \{i \in \mathcal{I} | F(j_{\overline{x_i}*}E(\overline{x_i})) = j_{\overline{x_i}*}E(\overline{x_i}) \},
\end{equation}
\begin{equation}
\mathcal{I}_2 = \{i \in \mathcal{I} | F(j_{\overline{x_i}*}E(\overline{x_i})) = 0 \}.
\end{equation}
By (1) we have $\mathcal{J} = \mathcal{J}_1\oplus \mathcal{J}_2.$ Since $\mathcal{T}_{F}$ is closed under taking direct sums and $\mathcal{F}_{F}$ is closed under taking direct products and subsheaves, we have $\mathcal{I}_1 \in \mathcal{T}_{F}$, and $\mathcal{I}_2 \in \mathcal{F}_{F}$. Therefore we have an equality $F(\mathcal{I}) = F(\mathcal{I}_1) \oplus F(\mathcal{I}_2) = \mathcal{I}_1$, which is an injective sheaf.
\end{proof}

\begin{proposition}
\label{youhanzi}
Let $F: QCoh(\mathfrak{X})\rightarrow QCoh(\mathfrak{X})$ be a left exact preradical functor which preserves injectivity and $x \in |\mathfrak{X}|$ be a point, then
\begin{itemize}
\item (1) $F(j_{\overline{x}*}E(\overline{x}))$ is identical to either $j_{\overline{x}*}E(\overline{x})$ or $0$.
\item (2) $F(j_{\overline{x}*}\kappa(\overline{x}))$ is identical to either $j_{\overline{x}*}\kappa(\overline{x})$ or $0$.
\end{itemize}
\end{proposition}
\begin{proof}
(1) Since $F(j_{\overline{x}*}E(\overline{x}))$ is an injective subsheaf of an indecomposable
injective sheaf $j_{\overline{x}*}E(\overline{x})$, it is a direct summand of $j_{\overline{x}*}E(\overline{x})$. Thus the
indecomposability of $j_{\overline{x}*}E(\overline{x})$ forces $F(j_{\overline{x}*}E(\overline{x}))$ is either  $j_{\overline{x}*}E(\overline{x})$ or $0$.

(2) It follows from Lemma \ref{prehanzijiao} that $F(j_{\overline{x}*}\kappa(\overline{x})) = j_{\overline{x}*}\kappa(\overline{x}) \cap F(j_{\overline{x}*}E(\overline{x}))$, therefore $F(j_{\overline{x}*}\kappa(\overline{x}))$ is either $j_{\overline{x}*}\kappa(\overline{x})$ or $0$ by (1). 
\end{proof}

For a left exact preradical functor $F$ which preserves injectivity, we define a subset $V_F$ of $|\mathfrak{X}|$ as follows:
\begin{equation}
V_F = \{ x \in |\mathfrak{X}| : F(j_{\overline{x}*}\kappa(\overline{x})) = j_{\overline{x}*}\kappa(\overline{x}) \}.
\end{equation}
Note from the proof of Proposition \ref{youhanzi} that $V_F$ is the same as the set 
\begin{equation}
\{ x \in |\mathfrak{X}| : F(j_{\overline{x}*}E(\overline{x})) = j_{\overline{x}*}E(\overline{x}) \}.
\end{equation}

\begin{proposition}\label{neishehanziqian}
Let $F$ be a left exact preradical functor which preserves injectivity.
Then $V_F$ is a specialization-closed subset.
\end{proposition}
\begin{proof}
Let $x \in V_F$ and $y\in \overline{\{x \}}$. A natural nontrivial morphism $\kappa(\overline{x}) \rightarrow \kappa(\overline{y}) \rightarrow  E(\overline{y}) $ extends to a non-zero morphism $j_{\overline{x}*}E(\overline{x}) \rightarrow j_{\overline{y}*}E(\overline{y})$. There is an exact sequence
\begin{equation}\nonumber
0 \rightarrow \mathcal{K} \rightarrow j_{\overline{x}*}E(\overline{x}) \rightarrow j_{\overline{y}*}E(\overline{y}) .
\end{equation}
Since $F$ is a left exact preradical functor, there is an exact sequence
\begin{equation}\nonumber
0 \rightarrow F(\mathcal{K}) \rightarrow F(j_{\overline{x}*}E(\overline{x})) \rightarrow F(j_{\overline{y}*}E(\overline{y})) .
\end{equation}
We have $F(j_{\overline{x}*}E(\overline{x})) = j_{\overline{x}*}E(\overline{x})$ and
$F(\mathcal{K}) = \mathcal{K} \cap F(j_{\overline{x}*}E(\overline{x})) = \mathcal{K}$ by Lemma \ref{prehanzijiao}, therefore $F(j_{\overline{y}*}E(\overline{y})) = j_{\overline{y}*}E(\overline{y})$, hence $y \in V_F$. 
\end{proof}

\begin{lemma}\label{jieyingyupre}
Let $F$ be a left exact preradical functor which preserves injectivity. Then the equality $F = \Gamma_{V_F}$ holds as subfunctors of $\mathbf{1}$, where $V_F$ is a specialization-closed subset of $|\mathfrak{X}|$ defined in Proposition \ref{neishehanziqian}.
\end{lemma}
\begin{proof}
First of all, we consider the case that $\mathcal{F}$ is a finite direct sum of indecomposabe injective objects $\bigoplus^{n}_{i=1}j_{\overline{x_i}*}E(\overline{x_i})$ in $QCoh(\mathfrak{X})$. 
Then we have an equality 
\begin{equation}\nonumber
F(\mathcal{F}) = \bigoplus_{{x_i}\in V_F} j_{\overline{x_i}*}E(\overline{x_i}) = \Gamma_{V_F}(\mathcal{F})
\end{equation}
by Proposition \ref{youhanzi}. 

Next, we consider the case that $\mathcal{F}\in Coh(\mathfrak{X})$. Since the injective hull $E(\mathcal{F})$ of $\mathcal{F}$ is a finite direct sum of indecomposable injective sheaves, we have already shown that $F(E(\mathcal{F})) = \Gamma_{V_F} (E(\mathcal{F}))$. Thus, using Lemma \ref{prehanzijiao}, we have 
\begin{equation}\nonumber
F(\mathcal{F}) = \mathcal{F}\cap F(E(\mathcal{F})) = \mathcal{F}\cap \Gamma_{V_F} (E(\mathcal{F})) = \Gamma_{V_F}(\mathcal{F}).
\end{equation}

Finally, we show the claimed equality for an object $\mathcal{F}$ in $QCoh(\mathfrak{X})$ without any assumption. We should notice that a coherent subsheaf $ \mathcal{G}\subset \mathcal{F}$ belongs to $F(\mathcal{F})$ if and only if the equality $F(\mathcal{G}) = \mathcal{G}$ holds. In fact, this equivalence is easily observed from the equality $F(\mathcal{G}) = \mathcal{G} \cap F(\mathcal{F})$ by Lemma \ref{prehanzijiao}. This equivalence is true for the section functor $\Gamma_{V_F}$ as well. So $\mathcal{G} \subset \mathcal{F}$ belongs to $\Gamma_{V_F}(\mathcal{F})$ if and only if $\Gamma_{V_F}(\mathcal{G})=\mathcal{G}$. Therefore, we see that $\mathcal{G} \subset F(\mathcal{F})$ if and only if $\mathcal{G} \subset \Gamma_{V_F}(\mathcal{F})$, and the proof is completed.
\end{proof}

\begin{proofA}
 $(1) \Rightarrow(2)$, $(2) \Rightarrow(3)$, $(3) \Rightarrow(1)$  are already proved respectively in Lemmas \ref{radicbaoneishe}(2), \ref{jieyingyupre} and \ref{jieying}(3).
\end{proofA}

Then we will also use the following functors. Suppose $V' \subset V$ is also a specialization-closed subset. Then, define $\Gamma_{V/V'} (\mathcal{F}) = \Gamma_V (\mathcal{F})/\Gamma_{V'} (\mathcal{F}).$

\begin{lemma}
Let $\mathfrak{X}$ be a Noetherian separated algebraic space over $S$, and $V' \subset V$ be specialization-closed subsets of $|\mathfrak{X}|$, then there is a right-derived functor $\mathbf{R}\Gamma_{V/V'}(-): \mathbf{D}(QCoh(\mathfrak{X})) \rightarrow \mathbf{D}(QCoh(\mathfrak{X}))$ such that the triangle $\mathbf{R}\Gamma_{V'}(\mathcal{F}^{\bullet}) \rightarrow \mathbf{R}\Gamma_{V}(\mathcal{F}^{\bullet}) \rightarrow \mathbf{R}\Gamma_{V/V'}(\mathcal{F}^{\bullet})\rightarrow \mathbf{R}\Gamma_{V'}(\mathcal{F}^{\bullet})[1]$ is distinguished.
\end{lemma}
\begin{proof}
By Lemma \ref{daochudengjia}, every quasi-coherent $\mathcal{O}_{\mathfrak{X}}$-complex has a $K$-injective resolution. It suffices to define $\mathbf{R}\Gamma_{V/V'}$ on $\mathbf{K}(\mathcal{I})$. Let $\mathcal{J}^{\bullet} \in \mathbf{K}(\mathcal{I})$ be quasi-isomorphic to $\mathcal{F}^{\bullet}$. Then the sequence 
\begin{equation}
0 \rightarrow \mathbf{R}\Gamma_{V'} (\mathcal{J}^{\bullet}) \rightarrow \mathbf{R}\Gamma_{V}(\mathcal{J}^{\bullet}) \rightarrow \mathbf{R}\Gamma_{V/V'} (\mathcal{J}^{\bullet}) \rightarrow 0
\end{equation}
is exact, and taking the corresponding distinguished triangle, we are done.
\end{proof}

\begin{lemma}\label{chouxiangyujieying}
Let $\mathcal{F}^{\bullet} \in \mathbf{D}(QCoh(\mathfrak{X}))$ and let $V$ be a specialization-closed subset of $|\mathfrak{X}|$. Then
\begin{itemize}
\item (1) $\mathcal{F}^{\bullet}$ belongs to $Im(\mathbf{R}\Gamma_V )$ if and only if
$\mathcal{F}^{\bullet}$ is quasi-isomorphic to an injective complex whose components are direct sums of $j_{\overline{x}*}E(\overline{x})$ with $x \in V$. 
\item (2) $\mathcal{F}^{\bullet}$ belongs to $Ker(\mathbf{R}\Gamma_V )$ if and only if
$\mathcal{F}^{\bullet}$ is quasi-isomorphic to an injective complex whose components are direct sums of $j_{\overline{x}*}E(\overline{x})$ with $x \in |\mathfrak{X}|-V$. 
\item (3) The natural embedding functor $i : Im(\mathbf{R}\Gamma_V) \rightarrow \mathbf{D}(QCoh(\mathfrak{X}))$ has a right adjoint $\rho : \mathbf{D}(QCoh(\mathfrak{X})) \rightarrow Im(\mathbf{R}\Gamma_V)$ and $\mathbf{R}\Gamma_V \cong i \circ \rho $. Hence, by Theorem \ref{miqi}, $(Im(\mathbf{R}\Gamma_V), Ker(\mathbf{R}\Gamma_V))$ is a \textit{stable $t$-structure} on $\mathbf{D}(QCoh(\mathfrak{X}))$.
\item (4) The $t$-structure $(Im(\mathbf{R}\Gamma_V), Ker(\mathbf{R}\Gamma_V))$ divides indecomposable injective quasi-coherent sheaves, by which we mean that each indecomposable injective quasi-coherent sheaf belongs to either $Im(\mathbf{R}\Gamma_V)$ or $Ker(\mathbf{R}\Gamma_V)$.
\end{itemize}
\end{lemma}
\begin{proof}
By Lemma \ref{daochudengjia}, every quasi-coherent $\mathcal{O}_{\mathfrak{X}}$-complex has a $K$-injective resolution.
For any injective complex $\mathcal{J}^{\bullet} \in \mathbf{K}(\mathcal{I})$, $\mathbf{R}\Gamma_V(\mathcal{J}^{\bullet}) = \Gamma_V(\mathcal{J}^{\bullet})$ is the subcomplex of $\mathcal{J}^{\bullet}$ consisting of injective objects supported in $V$. Hence every object of $Im(\mathbf{R}\Gamma_V)$ (resp. $Ker(\mathbf{R}\Gamma_V))$ is an injective complex whose components are direct sums of $j_{\overline{x}*}E(\overline{x})$ with $x \in V$ (resp. $x \in |\mathfrak{X}|-V$). In particular, if $x \in V$ (resp. $x \in |\mathfrak{X}|-V$), then $j_{\overline{x}*}E(\overline{x}) \in Im(\mathbf{R}\Gamma_V)$ (resp. $j_{\overline{x}*}E(\overline{x}) \in Ker(\mathbf{R}\Gamma_V)$). Since $Hom_{\mathcal{O}_{\mathfrak{X}}}(j_{\overline{x}*}E(\overline{x}), j_{\overline{y}*}E(\overline{y})) = 0$ for $x \in V$ and $y \in |\mathfrak{X}|-V$, we can see that $Hom_{\mathbf{K}(\mathcal{I})}(\mathcal{J}^{\bullet}_1,\mathcal{J}^{\bullet}_2) = Hom_{\mathbf{K}(\mathcal{I})}(\mathcal{J}^{\bullet}_1, \Gamma_V (\mathcal{J}^{\bullet}_2))$ for any $\mathcal{J}^{\bullet}_1 \in Im(\mathbf{R}\Gamma_V)$ and $\mathcal{J}^{\bullet}_2 \in \mathbf{K}(\mathcal{I})$. Hence it follows from the above equivalence that $\mathbf{R}\Gamma_V$ is a right adjoint of the natural embedding $i : Im(\mathbf{R}\Gamma_V) \rightarrow \mathbf{D}(QCoh(\mathfrak{X}))$.
\end{proof}

\section{Small supports}\label{Small supports}

\begin{definition}
The small support $supp(\mathcal{F}^{\bullet})$ of a complex $\mathcal{F}^{\bullet}$ of quasi-coherent sheaves on $\mathfrak{X}$ is defined as the set of points $x$ of $\mathfrak{X}$ satisfying $j_{\overline{x}*}\kappa(\overline{x}) \otimes^{\mathbf{L}} _{\mathcal{O}_{\mathfrak{X}}} \mathcal{F}^{\bullet} \not= 0 $ in $\mathbf{D}(QCoh(\mathfrak{X}))$.
\end{definition}

For a subset $W$ of $|\mathfrak{X}|$, we denote by $supp^{-1}(W)$ the subcategory of $\mathbf{D}(QCoh(\mathfrak{X}))$ consisting of all complexes $\mathcal{F}^{\bullet}$ of quasi-coherent sheaves on $\mathfrak{X}$ with $supp(\mathcal{F}^{\bullet}) \subset W$.

We define the subcategory $\mathcal{L}_W$ as the smallest localizing subcategory of $\mathbf{D}(QCoh(\mathfrak{X}))$ that contains the set  $\{j_{\overline{x}*}\kappa(\overline{x}): x \in W\}$. If $W=\{ x \}$, we will denote $\mathcal{L}_W$ simply by $\mathcal{L}_x$. Note that if $x\in W$, then $\mathcal{L}_x \subset \mathcal{L}_W$.

\begin{lemma}\label{shensheng}
Let $\mathcal{F}^{\bullet}$ be a complex that is quasi-isomorphic to an injective complex whose components are direct sums of $j_{\overline{x}*}E(\overline{x})$, then
$\mathcal{F}^{\bullet} \in \mathcal{L}_x.$
\end{lemma}
\begin{proof}
Recall that every element of $E(\overline{x})$ is annihilated by $\mathfrak{m}_{\overline{x}}^n$ for some $n$ (see \cite{Lam2012LecturesRings}, Remark 3.79). Thus $\mathcal{F}^{\bullet}$ has a filtration
\begin{equation}
0=\mathcal{F}^{\bullet}_0 \subset \mathcal{F}^{\bullet}_1 \subset \mathcal{F}^{\bullet}_2 \subset \ldots \subset \mathcal{F}^{\bullet},
\end{equation}
where $\mathcal{F}^{\bullet}_i$ is the subcomplex of $\mathcal{F}^{\bullet}$ and $j_{\overline{x}}^{-1}\mathcal{F}^{\bullet}_i$ is annihilated by $\mathfrak{m}_{\overline{x}}^i$ for each $i$. Then, $j_{\overline{x}}^{-1}(\mathcal{F}^{\bullet}_i/\mathcal{F}^{\bullet}_{i-1})$ is a complex of vector spaces over $\kappa(\overline{x})$ since it is annihilated by $\mathfrak{m}_{\overline{x}}$. 
Now from the short exact sequences $0 \rightarrow \mathcal{F}^{\bullet}_{i-1} \rightarrow \mathcal{F}^{\bullet}_i \rightarrow \mathcal{F}^{\bullet}_i/\mathcal{F}^{\bullet}_{i-1} \rightarrow 0$ we get distinguished triangles
$\mathcal{F}^{\bullet}_{i-1} \rightarrow \mathcal{F}^{\bullet}_i \rightarrow \mathcal{F}^{\bullet}_i/\mathcal{F}^{\bullet}_{i-1} \rightarrow \mathcal{F}^{\bullet}_{i-1}[1].$ By induction, this implies that $\mathcal{F}^{\bullet}_i \in \mathcal{L}_x$ for all $i$. But since $\mathcal{F}^{\bullet} = colim \mathcal{F}^{\bullet}_i$, and by Lemma \ref{holim}, a localizing subcategory is closed under direct limits, we see that $\mathcal{F}^{\bullet}$ is in $\mathcal{L}_x$.
\end{proof}

\begin{lemma}\label{suppxingzhi}
The small support has the following properties:
\begin{itemize}
\item (1) $supp_{\mathfrak{X}}(j_{\overline{x}*}\kappa(\overline{x}))= \{x \}.$ 
\item (2) Let $\mathcal{F}^{\bullet} \rightarrow \mathcal{G}^{\bullet} \rightarrow\mathcal{H}^{\bullet} \rightarrow\mathcal{F}^{\bullet}[1]$  be a distinguished triangle in $\mathbf{D}(QCoh(\mathfrak{X}))$. Then one has the following inclusion relations:
\begin{equation}\nonumber
supp(\mathcal{F}^{\bullet}) \subset supp(\mathcal{G}^{\bullet}) \cup supp(\mathcal{H}^{\bullet}),
\end{equation}
\begin{equation}\nonumber
supp(\mathcal{G}^{\bullet}) \subset supp(\mathcal{F}^{\bullet}) \cup supp(\mathcal{H}^{\bullet}),
\end{equation}
\begin{equation}\nonumber
supp(\mathcal{H}^{\bullet}) \subset supp(\mathcal{G}^{\bullet}) \cup supp(\mathcal{F}^{\bullet}).
\end{equation}
\item (3) The equality
\begin{equation}\nonumber
supp(\bigoplus_{i\in \mathcal{I}}\mathcal{F}^{\bullet}_i) =\bigcup_{i\in \mathcal{I}} supp(\mathcal{F}^{\bullet}_i) 
\end{equation}
holds for any family $\{ \mathcal{F}^{\bullet}_i \}_{i\in \mathcal{I}}$ of complexes of quasi-coherent sheaves on $\mathfrak{X}$.

\item (4) $supp_{\mathfrak{X}}(\mathcal{F}^{\bullet} \otimes^{\mathbf{L}} _{\mathcal{O}_{\mathfrak{X}}} \mathcal{G}^{\bullet})= supp(\mathcal{F}^{\bullet}) \cap supp (\mathcal{G}^{\bullet})$ for $\mathcal{F}^{\bullet},\mathcal{G}^{\bullet} \in \mathbf{D}(QCoh(\mathfrak{X}))$.

\item (5) $supp^{-1}(W)$ is a localizing subcategory of $\mathbf{D}(QCoh(\mathfrak{X}))$.

\item (6) Let $\mathcal{X}$ be a localizing subcategory of $\mathbf{D}(QCoh(\mathfrak{X}))$ and let $\mathcal{F}^{\bullet}\in \mathcal{X}.$ If $supp(\mathcal{F}^{\bullet})=W,$ then $\mathcal{L}_W$ is a subcategory of $\mathcal{X}$.

\item (7) $\mathcal{L}_W \subset supp^{-1}(W)$.

\item (8) Let $\mathcal{F}^{\bullet}$ be a complex in $\mathbf{D}(QCoh(\mathfrak{X}))$
that is quasi-isomorphic to an injective complex whose components are direct sums of $j_{\overline{x}*}E(\overline{x})$ with $x \in W$, then
$\mathcal{F}^{\bullet} \in \mathcal{L}_W.$ In particular, $\mathcal{L}_{|\mathfrak{X}|} = supp^{-1}(|\mathfrak{X}|) = \mathbf{D}(QCoh(\mathfrak{X}))$.

\item (9) $supp_{\mathfrak{X}}(\mathcal{F}^{\bullet})\not= \emptyset$ for every nontrivial $\mathcal{F}^{\bullet} \in \mathbf{D}(QCoh(\mathfrak{X}))$.

\item (10) Let $\mathcal{F}^{\bullet}$ be a complex of quasi-coherent sheaves on $\mathfrak{X}$ and $\mathcal{F}^{\bullet} \rightarrow \mathcal{I}^{\bullet}$ a minimal $K$-injective resolution of $\mathcal{F}^{\bullet}$. Then we have
\begin{equation}\nonumber
supp(\mathcal{F}^{\bullet}) = \bigcup_{i\in \mathbb{Z}}Ass(\mathcal{I}^{i}) .
\end{equation}

\item (11) $supp^{-1}(W)$ and the full subcategory of complexes that are quasi-isomorphic to injective complexes whose components are direct sums of $j_{\overline{x}*}E(\overline{x})$ with $x \in W$, are the same category.

\item (12) $supp_{\mathfrak{X}}(j_{\overline{x}*}E(\overline{x}))= \{x \}.$

\end{itemize}
\end{lemma}

\begin{proof}
(1) Let $x, y$ be points of $\mathfrak{X}$. Then we easily see that there are isomorphisms
\begin{equation}\nonumber
(j_{\overline{x}*}\kappa(\overline{x}) \otimes^{\mathbf{L}} _{\mathcal{O}_{\mathfrak{X}}} j_{\overline{y}*}\kappa(\overline{y}))_{\overline{x}}\cong
\kappa(\overline{x}) \otimes^{\mathbf{L}} _{\mathcal{O}_{\mathfrak{X},\overline{x}}} (j_{\overline{y}*}\kappa(\overline{y}))_{\overline{x}} 
\end{equation}
and
\begin{equation}\nonumber
(j_{\overline{x}*}\kappa(\overline{x}) \otimes^{\mathbf{L}} _{\mathcal{O}_{\mathfrak{X}}} j_{\overline{y}*}\kappa(\overline{y}))_{\overline{y}}
\cong (j_{\overline{x}*}\kappa(\overline{x}))_{\overline{y}} \otimes^{\mathbf{L}} _{\mathcal{O}_{\mathfrak{X},\overline{y}}} \kappa(\overline{y}).
\end{equation}
Therefore the complex $j_{\overline{x}*}\kappa(\overline{x}) \otimes^{\mathbf{L}} _{\mathcal{O}_{\mathfrak{X}}} j_{\overline{y}*}\kappa(\overline{y})$ is nonzero if and only if $x = y$.

(2) Let $x$ be a point in $supp(\mathcal{F}^{\bullet})$. Then $j_{\overline{x}*}\kappa(\overline{x}) \otimes^{\mathbf{L}} _{\mathcal{O}_{\mathfrak{X}}} \mathcal{F}^{\bullet} \not= 0 $. There is a distinguished triangle
\begin{equation}\nonumber
j_{\overline{x}*}\kappa(\overline{x}) \otimes^{\mathbf{L}} _{\mathcal{O}_{\mathfrak{X}}}\mathcal{F}^{\bullet} \rightarrow j_{\overline{x}*}\kappa(\overline{x}) \otimes^{\mathbf{L}} _{\mathcal{O}_{\mathfrak{X}}}\mathcal{G}^{\bullet} \rightarrow j_{\overline{x}*}\kappa(\overline{x}) \otimes^{\mathbf{L}} _{\mathcal{O}_{\mathfrak{X}}}\mathcal{H}^{\bullet} \rightarrow j_{\overline{x}*}\kappa(\overline{x}) \otimes^{\mathbf{L}} _{\mathcal{O}_{\mathfrak{X}}}\mathcal{F}^{\bullet}[1]
\end{equation}
which says that either $j_{\overline{x}*}\kappa(\overline{x}) \otimes^{\mathbf{L}} _{\mathcal{O}_{\mathfrak{X}}} \mathcal{G}^{\bullet}$ or $j_{\overline{x}*}\kappa(\overline{x}) \otimes^{\mathbf{L}} _{\mathcal{O}_{\mathfrak{X}}} \mathcal{H}^{\bullet}$ is nonzero. Thus $x$ is in the union of $supp(\mathcal{G}^{\bullet})$ and $supp(\mathcal{H}^{\bullet})$. The other inclusion relations are similarly obtained.

(3) One has $j_{\overline{x}*}\kappa(\overline{x}) \otimes^{\mathbf{L}} _{\mathcal{O}_{\mathfrak{X}}}(\bigoplus_{i\in \mathcal{I}}\mathcal{F}^{\bullet}_i) \cong \bigoplus_{i\in \mathcal{I}}  (j_{\overline{x}*}\kappa(\overline{x}) \otimes^{\mathbf{L}} _{\mathcal{O}_{\mathfrak{X}}} \mathcal{F}^{\bullet}_i)$ for
$x\in |\mathfrak{X}|$. Hence $j_{\overline{x}*}\kappa(\overline{x}) \otimes^{\mathbf{L}} _{\mathcal{O}_{\mathfrak{X}}}(\bigoplus_{i\in \mathcal{I}}\mathcal{F}^{\bullet}_i)$ is nonzero if and only if $j_{\overline{x}*}\kappa(\overline{x}) \otimes^{\mathbf{L}} _{\mathcal{O}_{\mathfrak{X}}} \mathcal{F}^{\bullet}_i$ is nonzero for some $i\in \mathcal{I}$.

(4) Notice that either $\mathcal{F}^{\bullet} \otimes^{\mathbf{L}} _{\mathcal{O}_{\mathfrak{X}}}j_{\overline{x}*}\kappa(\overline{x})=0$ or $\mathcal{G}^{\bullet} \otimes^{\mathbf{L}} _{\mathcal{O}_{\mathfrak{X}}}j_{\overline{x}*}\kappa(\overline{x})=0$ implies $\mathcal{F}^{\bullet}\otimes^{\mathbf{L}} _{\mathcal{O}_{\mathfrak{X}}}\mathcal{G}^{\bullet} \otimes^{\mathbf{L}} _{\mathcal{O}_{\mathfrak{X}}}j_{\overline{x}*}\kappa(\overline{x})=0$. Conversely, take $x \in supp(\mathcal{F}^{\bullet}) \cap supp (\mathcal{G}^{\bullet})$. Then $\mathcal{F}^{\bullet}\otimes^{\mathbf{L}} _{\mathcal{O}_{\mathfrak{X}}}\mathcal{G}^{\bullet} \otimes^{\mathbf{L}} _{\mathcal{O}_{\mathfrak{X}}}j_{\overline{x}*}\kappa(\overline{x}) = \mathcal{F}^{\bullet}\otimes^{\mathbf{L}} _{\mathcal{O}_{\mathfrak{X}}}(\bigoplus_{i} j_{\overline{x}*}\kappa(\overline{x})[i])= \bigoplus_{j} j_{\overline{x}*}\kappa(\overline{x})[j]$ which is nontrivial.

(5) This conclusion was immediately obtained by (2) and (3).

(6) The complex $j_{\overline{x}*}\kappa(\overline{x}) \otimes^{\mathbf{L}} _{\mathcal{O}_{\mathfrak{X}}} \mathcal{F}^{\bullet}$ is a direct sum of shifts of $j_{\overline{x}*}\kappa(\overline{x})$, hence if $x\in W,$ then $j_{\overline{x}*}\kappa(\overline{x})\in \mathcal{X},$ i.e. $\mathcal{L}_W$ is a subcategory of $\mathcal{X}$.

(7) This conclusion was immediately obtained by (6).

(8) (See also \cite{Neeman1992TheDR} Lemma 2.10 for an affine scheme $\mathfrak{X}$) Suppose $\mathcal{F}^{\bullet}$ is a complex in $\mathbf{D}(QCoh(\mathfrak{X}))$
that is quasi-isomorphic to an injective complex whose components are direct sums of $j_{\overline{x}*}E(\overline{x})$ with $x \in W$. Let $\mathbb{S}$ be the set of all specialization-closed subsets $Y$ such that $\mathbf{R}\Gamma_{Y}(\mathcal{F}^{\bullet})\subset \mathcal{L}_W.$ As localizing subcategories are closed under direct limits, $\mathbb{S}$ must be closed under the formation of increasing unions. Hence, by Zorn's Lemma, $\mathbb{S}$ contains a maximal element $Y$. We assert $Y = |\mathfrak{X}|$.

Suppose $Y \not= |\mathfrak{X}|$. Because $\mathfrak{X}$ is Noetherian, $|\mathfrak{X}| - Y$ contains an element $x$, such that
\begin{equation}\nonumber
\{x \}  = \overline{\{x \}} \cap (|\mathfrak{X}|-Y).
\end{equation}
But now we have
\begin{equation}\nonumber
\mathbf{R}\Gamma_{{Y\cup \{x\}}/Y}(\mathcal{F}^{\bullet})=\mathbf{R}\Gamma_{\overline{\{x \}}/{\overline{\{x \}}- \{x\}}}(\mathcal{F}^{\bullet})
\end{equation}
and $\mathbf{R}\Gamma_{\overline{\{x \}}/{\overline{\{x \}}- \{x\}}}(\mathcal{F}^{\bullet})\in \mathcal{L}_x \subset \mathcal{L}_W$ by Lemma \ref{shensheng}. Thus we deduce easily that $\mathbf{R}\Gamma_{{Y\cup \{x\}}}(\mathcal{F}^{\bullet}) \in \mathcal{L}_W,$ and this is a contradiction to the maximality of $Y$.

(9) If $supp_{\mathfrak{X}}(\mathcal{F}^{\bullet})= \emptyset,$ then by (8) we have $\mathcal{F}^{\bullet} = \mathcal{F}^{\bullet} \otimes^{\mathbf{L}} _{\mathcal{O}_{\mathfrak{X}}}\mathcal{O}_{\mathfrak{X}}=0.$

(10) (See also \cite{Krause2008ThickRings} Proposition 5.1 for an affine scheme $\mathfrak{X}$) Fix a point $x$. We know that 
\begin{equation}\nonumber
\mathcal{I}^{\bullet}\otimes^{\mathbf{L}} _{\mathcal{O}_{\mathfrak{X}}}j_{\overline{x}*}\kappa(\overline{x}) \cong  \mathbf{R}\Gamma_{\overline{\{x \}}/{\overline{\{x \}}- \{x\}}}(\mathcal{I}^{\bullet})\otimes^{\mathbf{L}} _{\mathcal{O}_{\mathfrak{X}}}j_{\overline{x}*}\kappa(\overline{x}).
\end{equation}

Suppose first that $\mathcal{I}^{\bullet}\otimes^{\mathbf{L}} _{\mathcal{O}_{\mathfrak{X}}}j_{\overline{x}*}\kappa(\overline{x}) \not=0$. Then $\mathbf{R}\Gamma_{\overline{\{x \}}/{\overline{\{x \}}- \{x\}}}(\mathcal{I}^{\bullet})\not=0$  and therefore $x \in Ass(\mathcal{I}^i)$ for some $i$.

Now suppose that $\mathcal{I}^{\bullet}\otimes^{\mathbf{L}} _{\mathcal{O}_{\mathfrak{X}}}j_{\overline{x}*}\kappa(\overline{x}) =0$. Then $\mathbf{R}\Gamma_{\overline{\{x \}}/{\overline{\{x \}}- \{x\}}}(\mathcal{I}^{\bullet})=0$ by (9), that is, $\mathcal{I}^{\bullet}\otimes^{\mathbf{L}} _{\mathcal{O}_{\mathfrak{X}}}j_{\overline{x}*}\kappa(\overline{x})$ is acyclic. We want to conclude that 
\begin{equation}\nonumber
\mathbf{R}\Gamma_{\overline{\{x \}}/{\overline{\{x \}}- \{x\}}}(\mathcal{I}^{i})=\mathbf{R}\Gamma_{\overline{\{x \}}/{\overline{\{x \}}- \{x\}}}(\mathcal{I}^{\bullet})^i=0
\end{equation}
for all $i$. Here we need to use the minimality of $\mathcal{I}^{\bullet}$. Also, $\mathbf{R}\Gamma_{\overline{\{x \}}/{\overline{\{x \}}- \{x\}}}(\mathcal{I}^{\bullet})$ is a $K$-injective complex of injective $\mathcal{O}_{\mathfrak{X}}$-modules. Thus $\mathbf{R}\Gamma_{\overline{\{x \}}/{\overline{\{x \}}- \{x\}}}(\mathcal{I}^{\bullet})=0$ in $\mathbf{D}(QCoh(\mathfrak{X}))$ implies $x\not\in Ass(\mathcal{I}^i)$ for all $i$, because $\mathcal{I}^{\bullet}$ is minimal.

(11) This conclusion was immediately obtained by (10).

(12) This conclusion was immediately obtained by (10).
\end{proof}

\section{Proof of Main Theorems}
\label{zhuyaozhangjie1}

\begin{theorem}
\label{miao}
Let $S$ be a scheme contained in $\mathbf{Sch}_{fppf}$. Let $\mathfrak{X}$ be a Noetherian separated algebraic space over $S$. Let $W \subset |\mathfrak{X}|$ be a subset. Then the following categories are the same:
\begin{itemize}
\item (1) $supp^{-1}(W).$
\item (2) $\mathcal{L}_W.$
\item (3) The full subcategory of complexes that are quasi-isomorphic to injective complexes whose components are direct sums of $j_{\overline{x}*}E(\overline{x})$ with $x \in W$. 
\end{itemize}
\end{theorem}
\begin{proof}
$(2)\subset (1).$ By Lemma \ref{suppxingzhi}(7).

$(1)= (3).$ By lemma \ref{suppxingzhi}(11).

$(3)\subset (2).$ By lemma \ref{suppxingzhi}(8).
\end{proof}

\begin{definition}
A localizing subcategory $\mathcal{X}$ is called rigid if for every $\mathcal{F}^{\bullet} \in \mathcal{X}$ and $\mathcal{G}^{\bullet} \in \mathbf{D}(QCoh(\mathfrak{X}))$, we have that $\mathcal{F}^{\bullet} \otimes^{\mathbf{L}} _{\mathcal{O}_{\mathfrak{X}}} \mathcal{G}^{\bullet} \in \mathcal{X}$.
\end{definition}

\begin{lemma}\label{xiaojieyinli}
For every subset $W \subset |\mathfrak{X}|$, the localizing subcategory $\mathcal{L}_W$ is rigid. 
\end{lemma}
\begin{proof}
By lemma \ref{suppxingzhi}(4).
\end{proof}

\begin{lemma}\label{xiaojieyinliya}
Let $\mathcal{X}$ be a rigid localizing subcategory of $\mathbf{D}(QCoh(\mathfrak{X}))$, and let $supp(\mathcal{X})=W$, then $\mathcal{X} = \mathcal{L}_W$ . 
\end{lemma}
\begin{proof}
It is obvious that $\mathcal{X}$ is contained in $\mathcal{L}_W=supp^{-1}(supp(\mathcal{X}))$. 

If $x\in W$, then there exists a complex $\mathcal{G}^{\bullet} \in \mathcal{X}$, such that $x\in supp(\mathcal{G}^{\bullet})$, i.e. $\mathcal{G}^{\bullet}\otimes^{\mathbf{L}} _{\mathcal{O}_{\mathfrak{X}}}j_{\overline{x}*}\kappa(\overline{x})\not=0$. The complex $\mathcal{G}^{\bullet}\otimes^{\mathbf{L}} _{\mathcal{O}_{\mathfrak{X}}}j_{\overline{x}*}\kappa(\overline{x})$ is a direct sum of shifts of $j_{\overline{x}*}\kappa(\overline{x})$, hence if $x\in W,$ then $j_{\overline{x}*}\kappa(\overline{x})\in \mathcal{X},$ i.e. $\mathcal{L}_W \subset \mathcal{X}$.
\end{proof}

\begin{theorem}\label{youyixiaojie}
One has maps
\begin{eqnarray*}
 \xymatrix{
  \mathbb{LS} \ar[rr]_{supp^{-1}} &  &  \mathbb{SX} \ar@< 2pt>[ll]_{supp} 
   }\\
    \end{eqnarray*}   
The map $supp$ is an inclusion-preserving bijection and $supp^{-1}$ is its inverse map.
\end{theorem}

\begin{proof}
By Lemma \ref{xiaojieyinli} the map $supp^{-1}$ is well-defined. Let $\mathcal{X}$ be a rigid localizing subcategory of $\mathbf{D}(QCoh(\mathfrak{X}))$. By Lemma \ref{xiaojieyinliya} we have $\mathcal{X} = \mathcal{L}_W,$ where $W=supp(\mathcal{X})$. Therefore $supp^{-1}(supp(\mathcal{X}))= supp^{-1}(W) = \mathcal{L}_W =\mathcal{X}$. Thus we conclude that the composite map $supp^{-1} \circ supp$ is the identity map.

Let $W$ be a subset of $|\mathfrak{X}|$. It is obvious that $supp(supp^{-1} (W))$ is contained in $W$. For $x \in W$ we have $supp (j_{\overline{x}*}E(\overline{x}))=\{x\} \subset W$ by Lemma \ref{suppxingzhi}(12). This implies that $W$ is contained in $supp(supp^{-1} (W))$, and we conclude that the composite map $supp \circ supp^{-1} $ is the identity map.
\end{proof}

\begin{definition}
Given a quasi-coherent sheaf $\mathcal{F}$,
\begin{equation}\nonumber
0 \rightarrow \mathcal{F} \rightarrow E^0 ( \mathcal{F} ) \xrightarrow{d^0} E^1 ( \mathcal{F} ) \xrightarrow{d^1} E^2 ( \mathcal{F} ) \rightarrow \ldots
\end{equation}
will be the minimal injective resolution of $\mathcal{F}$. 
We say that a subcategory $\mathcal{X}$ of $QCoh(\mathfrak{X})$ is $E$-stable provided that a sheaf $\mathcal{F}$ is in $\mathcal{X}$ if and only if so is $E^i(\mathcal{F})$ for every $i\geq 0$.
\end{definition}

For a subcategory $\mathcal{X}$ of $QCoh(\mathfrak{X})$, we denote by $supp(\mathcal{X})$ the set of points $x$ of $|\mathfrak{X}|$ such that $x \in supp (\mathcal{F})$ for some $\mathcal{F} \in \mathcal{X}$. 

For a subcategory $\mathcal{X}$ of $\mathbf{D}(QCoh(\mathfrak{X}))$, we denote by $\mathcal{X}_0$ the subcategory of $QCoh(\mathfrak{X})$ consisting of all quasi-coherent $\mathcal{O}_{\mathfrak{X}}$-sheaves $\mathcal{F}$ with $ \ldots \rightarrow 0 \rightarrow \mathcal{F} \rightarrow 0 \rightarrow 0 \rightarrow \ldots \in \mathcal{X}$.

\begin{proposition}\label{xiaoqinqin}
Let $W$ be a subset of $|\mathfrak{X}|$, and put $\mathcal{X} = (supp^{-1}(W))_0$.
\begin{itemize}
\item (1) $\mathcal{X}$ is closed under direct sums and summands.
\item (2) $\mathcal{X}$ is $E$-stable.
\end{itemize}
\end{proposition}
\begin{proof}
(1) We observe by Lemma \ref{suppxingzhi}(3) that $\mathcal{X}$ is closed under direct sums and summands. 

(2)Fix a sheaf $\mathcal{F}\in QCoh(\mathfrak{X})$. According to Lemma \ref{suppxingzhi}(10), we have
\begin{equation}
\begin{aligned}
\mathcal{F} \in \mathcal{X} &\Leftrightarrow supp(\mathcal{F})\subset W \\
&\Leftrightarrow AssE^{i}(\mathcal{F}) \subset W \ for \ all \ i\geq 0 \\
&\Leftrightarrow suppE^{i}(\mathcal{F}) \subset W \ for \ all \ i\geq 0 \\
&\Leftrightarrow E^{i}(\mathcal{F}) \in \mathcal{X} \ for \ all \ i\geq 0. \\
\end{aligned}
\end{equation}
Hence $\mathcal{X}$ is $E$-stable.
\end{proof}

\begin{theorem}\label{yuzhouxiaojie}
One has maps
\begin{eqnarray*}
 \xymatrix{
    \mathbb{SX}  \ar[rr]_{supp} &  &  \mathbb{ES} \ar@< 2pt>[ll]_{(supp^{-1}(-))_0}
   }\\
    \end{eqnarray*}   
The map $supp$ is an inclusion-preserving bijection and $(supp^{-1}(-))_0$ is its inverse map.
\end{theorem}
\begin{proof}
By Proposition \ref{xiaoqinqin} the map $(supp^{-1}(-))_0$ is well-defined. Let $\mathcal{X}$ be an $E$-stable subcategory of $QCoh(\mathfrak{X})$ closed under direct sums and summands. It is obvious that $\mathcal{X}$ is contained in $(supp^{-1}(supp(\mathcal{X})))_0$. Let $\mathcal{F}$ be a quasi-coherent $\mathcal{O}_{\mathfrak{X}}$-sheaf with $supp(\mathcal{F}) \subset supp(\mathcal{X})$. Then we see from Lemma \ref{suppxingzhi}(10) that for each $i \geq 0$ and $x \in AssE^i(\mathcal{F})$ there exists a sheaf $\mathcal{G} \in \mathcal{X}$ and an integer $j \geq 0$ such that $x \in AssE^j(\mathcal{G})$. Hence $j_{\overline{x}*}E(\overline{x})$ is isomorphic to a direct summand of $E^j(\mathcal{G})$. The sheaf $E^j(\mathcal{G})$ is in $\mathcal{X}$ since $\mathcal{X}$ is $E$-stable, and $j_{\overline{x}*}E(\overline{x})$ is also in $\mathcal{X}$ since $\mathcal{X}$ is closed under direct summands. Therefore the sheaf $E^{i}(\mathcal{F})$ is in $\mathcal{X}$ for every $i \geq 0$ since $\mathcal{X}$ is closed under direct sums, and $\mathcal{F}$ is also in $\mathcal{X}$ since $\mathcal{X}$ is $E$-stable. Thus we conclude that the composite map $(supp^{-1}(-))_0 \circ supp$ is the identity map.

Let $W$ be a subset of $|\mathfrak{X}|$. It is obvious that $supp((supp^{-1} (W))_0)$ is contained in $W$. For $x \in W$ we have $supp (j_{\overline{x}*}E(\overline{x}))=\{x\} \subset W$ by Lemma \ref{suppxingzhi}(12). This implies that $W$ is contained in $supp((supp^{-1} (W))_0)$, and we conclude that the composite map $supp \circ (supp^{-1}(-))_0 $ is the identity map.
\end{proof}

\bibliographystyle{plain}
\bibliography{references}

\end{document}